\documentclass[10pt]{amsart}

\usepackage{color}
\definecolor{dred}{rgb}{0.8,0,0.2}
\newcommand{\red}{\textcolor{dred}}
\definecolor{dblue}{rgb}{0.2,0,0.8}

\definecolor{dgreen}{rgb}{0.1,0.68,0.1}

\usepackage[all]{xy}
\usepackage{amssymb}
\usepackage{graphicx}

\usepackage{txfonts}
\usepackage{enumerate}
\usepackage{caption}
\usepackage{accents}

\usepackage{dsfont}

\def\comment#1{}

\newcommand{\op}{o}
\newcommand{\cl}{c}

\renewcommand{\ll}{\llbracket}

 \usepackage{amssymb}
\usepackage{amsmath}
\usepackage{amsthm}
\usepackage{amscd}
\usepackage{url}
\usepackage{color}

\theoremstyle{plain}
\newtheorem{them}{Theorem}[section]
\newtheorem{pro}[them]{Proposition}

\newtheorem{lem}[them]{Lemma}

\newtheorem{theorem}[them]{Theorem}

\newtheorem{lemma}[them]{Lemma}
\newtheorem{proposition}[them]{Proposition}

\theoremstyle{definition}
\newtheorem{defi}[them]{Definition}

\theoremstyle{remark}
\newtheorem{rem}[them]{Remark}
\newtheorem{ex}[them]{Example}

\DeclareMathOperator{\supp}{supp}

\newcommand{\HL}{Henry-Labord{\`e}re }

\theoremstyle{remark}
\newtheorem{remark}[them]{Remark}
\newtheorem{example}[them]{Example}

\DeclareMathOperator{\id}{Id} 
 
\DeclareMathOperator{\proj}{proj}

\DeclareMathOperator{\spt}{spt}

\DeclareMathOperator{\law}{Law}

\newcommand{\eps}{\varepsilon}
\newcommand{\Lg}{\lambda}
\renewcommand{\P}{\mathbb{P}}
\newcommand{\p}{\mathcal{P}}
\newcommand{\M}{\Pi_M} 
\newcommand{\hM}{\hat{\Pi}_M} 
\newcommand{\m}{\mathcal{M}}
\newcommand{\R}{\mathbb{R}}

\newcommand{\Q}{\mathbb{Q}}
\newcommand{\N}{\mathbb{N}}

\newcommand{\dd}{\mathrm{d}}

\newcommand{\E}{\mathbb{E}}

\newcommand{\sto}{\mathrm{sto}}

\newcommand{\leqp}{\preceq_+} 
\newcommand{\leqc}{\preceq_{C}} 
\newcommand{\leqs}{\preceq_\sto} 
\newcommand{\leqcp}{\preceq_{C,+}} 
\newcommand{\leqe}{\preceq_{C,+}} 
\newcommand{\lc}{\mathrm{lc}}
\newcommand{\sun}{\mathrm{sun}}

\newcommand{\I}{\mathds 1}

\newcommand{\spac}{[0,1]\times \R\times \R}
\newcommand{\spa}{[0,1]\times \R}

\renewcommand{\subset}{\subseteq}
\renewcommand{\supset}{\supseteq}

\title{Shadow couplings}
\author{Mathias Beiglb\"ock and Nicolas Juillet}

\thanks{The first author acknowledges support through FWF grant Y782. The second author was partially supported by the ``Programme ANR JCJC GMT'' (ANR 2011 JS01 011 01).  }


\begin{document}
\begin{abstract}
A classical result of Strassen asserts that given probabilities $\mu, \nu$ on the real line which are in convex order, there exists a \emph{martingale coupling} with these marginals, i.e.\ a random vector $(X_1,X_2)$ such that  $X_1\sim \mu, X_2\sim \nu$ and $\E[X_2|X_1]=X_1$. 
 Remarkably, it is a non trivial problem to construct particular solutions to this problem. 
 In this article, we introduce a family of such  martingale couplings, each of which admits several characterizations in terms of optimality properties / geometry of the support set / representation through a Skorokhod embedding. 
As a particular element of  this family we recover
the (left-) curtain martingale transport, which has recently been studied  \cite{BeJu16, HeTo13, CaLaMa14, BeHeTo15} and which can be viewed as a martingale analogue of the classical monotone rearrangement.
As another canonical element of  this family we identify a martingale coupling that resembles the usual \emph{product coupling} and appears as an optimizer in the general transport problem recently introduced by Gozlan et al. 
In addition, this coupling provides an explicit example of a Lipschitz-kernel, shedding new light on Kellerer's proof of the existence of Markov martingales with specified marginals.

\medskip


\end{abstract}
\maketitle


\section{Introduction}


Given Polish spaces $X,Y$, a measure $\pi$ on $X\times Y$ with marginals $\mu$ and $\nu$ is called a {\it transport plan} from $\mu$ to $\nu$ or a {\it coupling} of $\mu$ and $\nu$. Let $\Pi(\mu,\nu)$ be the space of transport plans of marginals $\mu$ and $\nu$.
We will usually consider probability measures $\mu$, $\nu$ on the real line having first moments in which case the set of \emph{martingale transport plans} is defined as 
\begin{align}\M(\mu,\nu)&\,=\{\pi=\law(X,Y)\in\Pi(\mu,\nu),\,\E(Y|X)=X\}\\
&\, =\textstyle \{\pi\in \Pi(\mu,\nu): \int y\, \dd\pi_x=y \mbox{ for $\hat\mu$-a.e.\ $(u,x)$ }\}.\end{align}
Here the constraint $\E(Y|X)=X$ means that $\E(Y|X=x)=x$ for $\mu$-almost every $x\in \R$, while $(\pi_x)_x$ denotes the disintegration of $\pi$ wrt\ $\mu$. 
By a classical result of Strassen \cite{St65}, $\M(\mu,\nu)$ is non-empty if and only if
 $\mu, \nu$ are in the convex order $\mu\leqc \nu$, i.e.\ both measures have finite first moments and $\int \phi\, \dd\mu\leq \int \phi\, \dd\nu$ for every convex $\phi:\R\to \R$.

In \cite{BeJu16} we introduced the (left-) curtain coupling $\pi_{\lc}$ which can be seen as a martingale analogue of the monotone rearrangement coupling. An  explicit description is provided  when $\mu$ is finitely supported (\cite[Chapter 2]{BeJu16}). Another construction using differential equations is given by \HL and Touzi \cite{HeTo13} for  sufficiently regular distributions.  According to \cite{Ju_ihp} the coupling method is continuous so that all left-curtain couplings for general measures $\mu$ and $\nu$ can be approximated using either  of the two mentioned constructions, see \cite[Remark 2.18]{Ju_ihp}.

  In this paper we will view $\pi_{\lc}$ as one extreme of an infinite family of martingale couplings whose construction is based each on a different parametrization of $\mu$: the curtain coupling will be recovered as one  `end' of this family for a \emph{horizontal} parametrization of $\mu$ (a curtain closed from left to right), while at the other end of the spectrum using a \emph{vertical} parametrization we will obtain a new and rather different type of canonical coupling that we shall call sunset coupling $\pi_{\sun}$.  This coupling can be seen as the martingale analogue of the product coupling $\mu\times \nu$ (in classical optimal transport). 
  In view of this it is natural that $\pi_{\sun}$ does not appear as an optimizer of the martingale version of the transport problem. However, we shall see in Theorem \ref{MainTheorem} and Section 5 below that enjoys some optimality properties of a different type. 
 A further particular property is that $\pi_{\sun}$ yields an explicit example of a martingale transport plan which has the \emph{Lipschitz(-Markov) property}. The existence of a martingale transport plan with this property is a key ingredient in all (to the best of our knowledge) proofs of Kellerer's Theorem \cite{Ke72} on the existence of Markov martingales with given marginals. Previous constructions of such transport plans were either non-constructive or relied on particular solutions to the Skorokhod embedding problem.

\subsection{Notation and Main Results}\label{subsec:main_result}
We write $\lambda$ for the Lebesgue measure on the unit interval and assume that $\mu\leqc \nu$. We  fix a \emph{lift} of $\mu$, that is,  a probability $\hat \mu\in \Pi(\lambda, \mu)$ that will serve as a parameter in the construction of a general version of the left-curtain coupling. 
The set of \emph{lifted martingale transport plans} is 
$$\hM(\hat \mu, \nu):=\textstyle\left\{\hat \pi\in \Pi(\hat \mu, \nu): \int y\,  \dd\hat\pi_{u,x}=x \mbox{ for $\hat\mu$-a.e.\ $(u,x)$ }\right\},$$ 
where $(\pi_{u,x})$ denotes the disintegration of  $\hat \pi $ wrt\ $\hat \mu$.



\begin{align*}
\xymatrix @!=2cm
{
{}&\boxed{\text{``lifted'' } \hat{\theta}} \ar[ld]^-{\txt{disintegration}}\ar[rd]^-{\txt{$\theta_{[0,u]}(A)=\hat{\theta}([0,u]\times A)$}}&{}\\
\boxed{(\theta_u)_{u\in[0,1]}} \ar[ru]^-{\txt{integration}} \ar[rr]^-{\txt{primitive}} &{}&\boxed{(\theta_{[0,u]})_{u\in[0,1]}} \ar[lu] \ar[ll]^-{\txt{derivative}}
}
\end{align*}

We shall use two further ways to denote the objects $\hat \mu\in \Pi(\lambda, \mu)\subset \p([0,1]\times \R)$ and $ \hat \pi\in \Pi(\hat \mu, \nu)\subset \p([0,1]\times \R^2)$: Given a measure $\hat \theta $ on $[0,1]\times \R^d$ (where $d=1,2$), with $\proj_{[0,1]}(\hat \theta)=\lambda$, we  write 
\begin{enumerate}
\item $(\theta_u)_{u\in [0,1]}$ for the ($\lambda$-a.s.\ unique) disintegration of $\hat \theta $ wrt\ $\lambda$.
\item $(\theta_{[0,u]})_{u\in [0,1]}$ for the family of measures defined for every $u\in[0,1]$ by
$$\theta_{[0,u]}(A)=\hat{\theta}([0,u]\times A)=\int_0^u\theta_{s}(A)\,ds.$$
\end{enumerate}
Our main result is the following.
\begin{theorem}\label{MainTheorem}
Let $\mu, \nu$ be real probability measures in convex order and $\hat \mu \in \Pi(\lambda, \mu)$. There exists a unique $\hat \pi\in \hat{\Pi}_M(\hat \mu, \nu)$ satisfying any, and then all of the following properties:
\begin{enumerate}
\item $\hat \pi$ minimizes 
\begin{align}
\hat \gamma \mapsto \int (1-u)\sqrt{1+y^2} \, \dd\hat \gamma
\end{align}
on the set $\hM(\hat \mu, \nu)$. 
\item $\hat \pi=\law(U, B_0,B_\tau)$, where 
$(B_t)$ is one dimensional Brownian motion, $\law (U, B_0)=\hat \mu$ and $\tau$ is the hitting  time of the  process $t\mapsto (U, B_t)$ into a \emph{left barrier} (i.e.\ a Borel set $R\subset \spa$ such that $(u, x)\in R, v\leq u$ implies $(v,x)\in R$).
\item  $\hat\pi(\hat \Gamma)= 1$ for a Borel set $\hat \Gamma\subseteq \spac$ which is \emph{monotone} in the sense that for all $s,t,x,x',y^-,y^+, y'$ 
$$ s< t, (s,x,y^-), (s,x,y^+), (t,x',y')\in \Gamma \Rightarrow y'\notin ]y^-,y^+[.$$ 
\item For all $u\in [0,1]$, the projection of $\pi_{[0,u]}$ onto the second coordinate is the \emph{shadow} of $\mu_{[0,u]}$ onto the measure $\nu$.
\end{enumerate}
\end{theorem}
We add some comments to this result:
\begin{itemize}
\item To make sense of the last point, note that if $\mu'\leq \mu$ and $\mu\leqc\nu$, then the set $\{\nu': \mu'\leqc \nu'\text{ and }\nu'\leq \nu\}$ is non-empty and has a smallest element $S^\nu(\mu')$ wrt\ $\leqc$, the \emph{shadow} of $\mu'$ onto the measure $\nu$ (cf.\ \cite[Lemma 4.6]{BeJu16} / Definition \ref{def:shadow} below). Intuitively speaking, among all measures $\nu'\leq \nu$ which are larger than $\mu'$ in convex order, $S^\nu(\mu')$ is the most concentrated one.
\item 
We will see below in Proposition \ref{Monotone2Barrier} that $\hat \pi \in \hM(\hat \mu, \nu)$ implies in the setting of (2) that the martingale $(B_t^\tau)_{t\geq 0}$ is uniformly integrable. 

\item In (1) the cost $c:(u,x,y)\mapsto (1-u)\sqrt{1+y^2}$ can be replaced by any positive $c(u,x,y)=\varphi(u)\psi(y)$ where $\varphi\geq 0$ is strictly decreasing and $\psi\geq 0$ is strictly convex and the minimum over $\hM(\hat \mu, \nu)$ is finite. More generally the same conclusions holds for a non-negative $c$ with $c_{uyy}<0$ in a weak sense. Alternative assumptions to  $c\geq0$ are that  $\int |\varphi| \,\dd\lambda,\int |\psi(y)|\, \dd \nu < \infty$
or that $c(x,y)\geq A+Bx+Cy$. 

\end{itemize}


We call the unique element of $\hM(\hat\mu,\nu)$ characterized in Theorem \ref{MainTheorem} the \emph{lifted shadow coupling} (with lift $\hat \mu$). 
Its projection onto the two last coordinates is an element $\pi$ of $\M(\mu,\nu)$ that we call \emph{shadow coupling} of $\mu$ and $\nu$ associated to the lift $\hat\mu$. Note that $\pi=\pi_{[0,1]}$ in the terminology introduced in Section 1.1. 

If the lift $\hat\mu$ is concentrated on the graph of a 1-1 function $T:[0,1]\to \R$ there is an obvious correspondence between elements of $\M(\mu,\nu)$ and $\hM(\hat\mu,\nu)$. In particular the optimality property stated in  Theorem \ref{MainTheorem} (1) then translates to an optimality property for the martingale version of the transport problem; early papers to investigate such problems include \cite{HoNe12,  BeHePe12, GaHeTo12, DoSo12, BoNu13, HoKl13, CaLaMa14, BeCoHu14, BeNuTo15}. 
For general lifts, the shadow coupling  does not exhibit particular optimality properties for the martingale transport problem.   However, it is characterized by a general optimality problem in the sense of Gozlan et al. \cite{GoRoSaTe14}. We shall discuss this in Section 5 below.

Canonical choices of lifts lead to canonical martingale couplings of shadow type. We shall be particularly interested in the cases where $\hat \mu$ is either the quantile or the product coupling of $\Lg$ and $\mu$:

\begin{itemize}
\item The quantile coupling $\Lg$ and $\mu$ is the unique coupling $\hat\mu$ whose support is an increasing function (which then is of course the quantile function of $\mu$). Considering the corresponding lifted shadow coupling in $\hM(\hat \mu, \nu)$, we recover the \emph{left-curtain coupling} introduced in \cite{BeJu16}. We shall henceforth denote this coupling by $\pi_\lc$. 

Notably most of the results established for $\pi_\lc$ in \cite{BeJu16} are a particular consequence of Theorem \ref{MainTheorem}
 (cf.\ Remark \ref{FamousCoro}). 

\item The other shadow coupling we will be particularly interested in  is our new \emph{sunset coupling} $\pi_\sun$. It is based on the product lift $\hat{\mu}=\Lg\otimes \mu$, i.e.\ the independent coupling of $\Lg$ and $\mu$. \end{itemize}

Note that the sources of the above martingale couplings are the two most natural coupling methods for elements in the space without constraint $\Pi(\mu,\nu)$. Looking at the measure $\mu$ as the hypograph of its unit density function, we note that the curves $(\mu_{[0,u]})_{u\in[0,1]}$ and $(\mu_u)_{u\in[0,1]}$ reminds a curtain closed from the left in the first case and from the bottom in the second case. This motivates the names \emph{curtain coupling} and \emph{sunset coupling}. The lifted versions are naturally denoted by $\hat\pi_\lc,\,\hat\pi_\sun$ and called \emph{lifted curtain coupling}, \emph{lifted sunset coupling} respectively.

\subsection{Kellerer's theorem and the sunset coupling}\label{KelSun}

Concerning the continuous setting, 

Kellerer's Theorem \cite{Ke73} states that if a family of measures $(\mu_t)_{t\in \R_+}$ satisfies 
$$s\leq t\Rightarrow \mu_s\leqc \mu_t,$$
there exists a martingale $(X_t)_{t\in \R_+}$ with $\law(X_t)=\mu_t$ for every $t$. The martingale can be supposed \emph{Markovian} and this is, as far as we are concerned, the most spectacular achievement of this theorem. In contemporary terms (see \cite{HPRY}),  $(\mu_t)_{t\in \R_+}$ is called a \emph{peacock} and $(X_t)_{t\in \R_+}$ is a Markovian martingale \emph{associated} to this peacock. 

To our best knowledge, all proofs of Kellerer's theorem  are based on approximation arguments using sequences of Markov processes. Here the
 obstacle is that the Markov-property is {\emph{not}} preserved when passing to the limit. The key insight of Kellerer was to consider   Lipschitz-Markov processes, that is Markov-processes 
 whose transition kernels have the following \emph{Lipschitz}-property: a kernel $P: x\mapsto \pi_x $ is called Lipschitz (or more precisely $1$-Lipschitz) if  $W(\pi_x,\pi_{x'})\leq |x-x'|$ for all $x,x'$ (cf.\ \eqref{kanto} for the precise definition of the Wasserstein-$1$ distance). It is then not difficult to see that the property of being a Lipschitz-Markov process \emph{is} preserved  when passing to the limit in the sense of finite-dimensional distributions.
 
 Martingale transport plans can be seen as a one step martingales and it is possible to compose several of them for defining a discrete Markovian martingale. The main technical step in Kellerer's proof is therefore to show that given measures $\mu, \nu$ in convex order there exists a martingale transport plan which has the Lipschitz property. 
  
 Let us add that Lipschitz kernels for measures in convex orders do not exist in higher dimensions $d\geq 2$ (see for instance \cite{Ju_seminaire}). As Lipschitz kernels and their avatars are the only known methods for proving the Kellerer theorem, still to our best knowledge, it is an open problem whether this theorem holds in dimensions greater than or equal to two. 
 
 While non of the various extremal martingale couplings constructed in \cite{HoNe12, BeJu16, HoKl13, CaLaMa14, St14} has the Lipschitz property, we shall see that the sunset coupling has the Lipschitz-property. Moreover, it connects to Kellerer's original proof (in \cite{Ke72}) of the existence of Lipschitz kernels which we now describe:


In Kellerer's  terminology (\cite{Ke72, Ke73}), 
 martingale transport plans $\pi$ appear as pairs consisting  of an initial measure $\mu$ together with a transition kernel $P:x\mapsto \pi_x$, satisfying $\int y\,\dd \pi_x(y)=x$ and $\mu P=\nu$. Slightly abusing notation, we will occasionally  identify martingale transport plans with their kernels. 
The strategy of Kellerer in order to find a Lipschitz kernel is to use Choquet's theory for describing $\nu$ as a combination of extreme measures of $E(\mu)$, the set of measures greater than $\mu$ in the convex order. He proves that if $\omega$ is an extreme element, the set $\M(\mu, \omega)$ consists only of a single element $\pi$ and this element $\pi$ has the Lipschitz property (see Proposition \ref{unique}).
Using that the Lipschitz-property is preserved under convex combinations, Kellerer then obtains that $\M(\mu,\nu)$ contains some Lipschitz kernel.


Let us now discuss how $\pi_\sun\in\M(\mu,\nu)$ compares to Kellerer's proof of the existence of Lipschitz kernels. In fact according to the barrier characterization in Theorem \ref{MainTheorem}, we have
$$\pi_\sun=\int_0^1\hat\pi_u \,\dd u.$$
 Here $\hat\pi_u=\law(B_0,B_\tau\mid\,U=u)$ where $(B_t)$ is a Brownian motion with starting distribution $\mu$, and $\tau$ (conditioned on $U=u$) is the hitting time of the barrier's vertical section $R_u:=\{y\in\R:\,(u,y)\in R\}$, see Figure \ref{barrier}. We denote  by $\mu P_{R_u}$ the hitting law $\law(B_\tau\mid\,U=u)$ and call the corresponding kernel $P_{R_u}$ the \emph{hitting projection} onto $R_u$ (see also Definition \ref{def:hitting}). We have $\nu=\int_0^1 \mu P_{R_u}\,\dd u$. 
 
\begin{figure}[ht]
\begin{center}
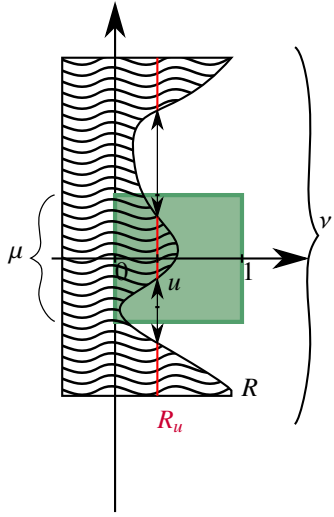
\caption{Hitting the barrier}\label{barrier} 
\end{center}
\end{figure}
 
  Kellerer established  that the measures $\omega=\mu P_T$ obtained through a hitting projection are exactly the extreme elements of $E(\mu)$ (provided one obtains a \emph{martingale} transport plan, that is $\spt(\mu)\subset[\inf(T),\sup(T)])$. Recalling the argument two paragraphs above, $\hat{\pi}_u=\mu(\id\times P_{R_u})$ is the unique element of $\M(\mu,\mu P_{R_u})$, the kernels $P_{R_u}$ are Lipschitz kernels, and in particular, the sunset coupling has the Lipschitz property. 
  

 We stress that the uniqueness in Theorem \ref{MainTheorem} permits us to gurantee that there exists a \emph{unique} Choquet representation through a family $(R_u)_{u\in[0,1]}$ which is \emph{ordered} in the sense that $s\leq t$ implies that $R_t\supset R_u$. Indeed, the definition of barrier tells us that $R_v\subset R_u$ if $u\leq v$ so that such a representation exists. If some Choquet representation of $\nu$ is given by measures $\nu_u$ obtained using the  hitting projection on sets $R_u$ which is ordered in the above sense, then these sets consitute a barrier and based on the uniqueness assertion in Theorem \ref{MainTheorem}, the family $(\nu_u)_{u\in[0,1]}$ with $\nu=\mu P_{R_u}$ is the one associated to the sunset coupling. 


Let us summarize this in a theorem. 

\begin{them}[Sunset coupling and Lipschitz kernel]\label{Kellerer-Choquet}
Let $\mu$ and $\nu\in \p$ satisfy $\mu\leqc \nu$. Then there exists a probability measure $\chi$ on $\mathcal{I}=\{T\subset \R\,:\,T\text{ is closed},\,\inf T=-\infty,\,\sup T=+\infty\}$ which represents $\nu$ in the sense that   
\begin{align}\label{ChoquetSense}\textstyle \nu(A)=\int_\mathcal{I} \mu P_T(A)\, \dd \chi(T)=:\mu P_\chi(A)=\mu_\chi(A).\end{align}
Moreover $P_\chi$ is a Lipschitz kernel and $\mu(\id\times P_\chi)\in \M(\mu,\nu)$ is a Lipschitz martingale plan.

A possible choice of $\chi$ is the uniform measure on $(R_u)_{u\in[0,1]}$ where $R_u$ is a decreasing family for $\subset$. Moreover if $\chi'$ is another measure associated to a non-increasing family $(R'_u)_{u\in[0,1]}$ with $\nu=\mu P_{\chi'}$ and $\mu(\id\times P_{\chi'})\in \M(\mu,\nu)$, then $\mu(\id\times P_\chi)=\mu(\id\times P_{\chi'})$. (This measure is the sunset coupling of $\mu$ and $\nu$, that is the shadow coupling of $\mu$ with lift $\lambda\times \mu$ and $\nu$.)
\end{them}


Another proof that $\M(\mu,\nu)$ contains at least one Lipschitz kernel is given by solutions of the Skorohod embedding problem. Root's embedding is considered in \cite{BeHuSt15} and Hobson's embedding in \cite[Lemma 3.3]{Low2}. To our best knowledge, this is after Kellerer's original proof the second known type of proof. Theorem \ref{MainTheorem} finally spans a bridge between the two methods because it presents a kind of lifted Skorohod embedding.

%
%
%

\section{Preliminaries}

\subsection{Definitions, notations on the martingale transport problem}

We consider the space $\m$ of positive measures on $\R$ with finite first moments. The subspace of probability measures with finite expectations is denoted by $\p$. In higher dimensions we denote the corresponding spaces by $\m(\R^d)$ and $\p(\R^d)$. For $\mu,\,\nu\in\m$, the Wasserstein-$1$ distance is defined by
\begin{align}\label{kanto}
W(\mu,\nu)=\sup_{f\in\mathrm{Lip}(1)}\left|\int\,f\,\dd\mu-\int\,f\,\dd\nu\right|
\end{align} and
endows $(\p,W)$ with $\mathcal{T}_1$, the usual topology for probability measures with finite first moments. In the definition, the supremum is taken over all $1$-Lipschitzian functions $f:\R\to\R$. We also consider  $W$ (with the same definition) on the subspace $m\p=\{\mu\in\m|\,\mu(\R)=m\}\subset\m$ of measures of mass $m$.

According to the Kantorovich duality theorem, an alternative definition in the case $\mu,\nu\in\p$ is
\begin{align}\label{wass}
\inf_{(\Omega,X,Y)}\E(|Y-X|)
\end{align}
where $X,\,Y:(\Omega,\mathcal{F},\P)\to\R$ are random variables with marginals $\mu$ and $\nu$. The infimum is taken among all joint laws $(X,Y)$, the probability space $(\Omega,\mathcal{F},\P)$ being part of the minimisation problem. Note that without loss of generality $(\Omega,\mathcal{F},\P)$ can be assumed to be $([0,1],\mathcal{B},\lambda)$ where $\lambda$ is the Lebesgue measure and $\mathcal{B}$ the $\sigma$-algebra of Borel sets on $[0,1]$.

A special choice of 1-Lipschitzian function is the function $f_t:x\in\R\to |x-t|\in \R$. Therefore if $\mu_n\to\mu$ in $\m$, the sequence of functions $u_{\mu_n}:t\mapsto \int f_t(x)\, \dd\mu_n(x)$  converges to $u_\mu$ pointwise. The converse statement also holds if all the measures have the same mass and barycenter (see \cite[Proposition 2.3]{HiRo12} or  \cite[Proposition 4.2]{BeJu16}). For $\mu\in\m$, the function $u_\mu$ is usually called the \emph{potential function} of $\mu$. 

\subsection{Bijection between curves, primitive curves, and lifted measures}

We elaborate on the equivalent avatars of the lifted measures introduced in Subsection \ref{subsec:main_result}. In short,  we are representing the same mathematical object in three ways: we consider the measure $\hat{\theta}$,  the almost surely defined disintegration  $(\theta_u)_{u\in[0,1]}$, and the primitive curve $(\theta_{[0,u]})_{u\in[0,1]}$. We first recall the integrability conditions. The lifted measure $\hat{\theta}$ is a probability measure on $[0,1]\times \R^d$ (where $d=1$ or $d=2$) such that 
$\hat{\theta}(\hat \rho)<+\infty$ is finite where $\rho:x\mapsto \|x\|_{\R^d}$ and $\hat{\rho}(u,x)=\rho(x)$. This integrability condition corresponds to $\theta_{[0,1]}(\rho)<+\infty$ for the primitive curve $(\theta_{[0,u]})_{u\in[0,1]}$ and $\int_0^1 \theta_u(\rho)\,\dd u<+\infty$ for $(\theta_{u})_{u\in[0,1]}$. The marginal condition asserts that $\hat\theta\in \Pi(\lambda,\theta)$ for some $\theta\in\p(\R^d)$. 
In terms of the primitive curve this can be expressed by asserting that $\theta_{[0,1]}=\theta$ and $\theta_{[0,u]}(\R)=u$.
The equivalent condition on $(\theta_u)_{u\in[0,1]}$ is that $\lambda$-almost surely $\theta_u\in\p$ and $\theta=\int_0^1 \theta_u\,\dd u$. Note that from a probabilistic point of view, if $(U,X)$ is a random vector of law $\hat{\theta}$ with $U\sim \lambda$, the other representations are given by $(u\times\law(X|U\leq u))_{u\in[0,1]}$ and $(\law(X|\,U=u))_{u\in[0,1]}$. Finally the object that we will Ultimately be most interested in is not the lift $\hat{\theta}$ but $\theta:=\theta_{[0,1]}=\law(X)$.

In what follows we explain that the derivative of the primitive  curve $(\theta_{[0,u]})_{u\in [0,1]}$ can be considered with respect to $\mathcal{T}_1$, the weak topology wrt to continuous functions which have at most linear growth.
Let us start with $\hat{\theta}\in \Pi(\lambda,\theta)$. We disintegrate the measure with respect to the first marginal and obtain an a.s.\ uniquely determined family $(\theta_u)_{u\in[0,1]}$ such that for almost every $u$, $\theta_u\in \p(\R^d)$. (We can assume that the measure is zero for the other parameters.)
Define $\theta_{[0,u]}$  for $u\in[0,1]$ by
$$\textstyle \theta_{[0,u]}(A)=\hat{\theta}([0,u]\times A)=\int_0^u\theta_{s}(A)\,\dd s$$
for  $A\subset \R^d$  Borel. Given a  function $f:x\in\R^d\to \R$ with $f(x)/(1+\|x\|)$ bounded, the function
$s\mapsto \theta_{s}(f)$ is measurable and in $L^1([0,1])$. Hence at almost every time $u\in[0,1]$ the  function $t\mapsto \theta_{[0,u]}(f)=\int_0^u\theta_{t}(f)\,\dd t$ is differentiable with derivative $\theta_u(f)$. It is important that the set $L\subset [0,1]$ of times at which the derivative exists for all $f$ is a Borel set of full measure, as we will verify in the next paragraph. Before establishing this claim, note that this permits us to define a canonical disintegration $(\tilde{\theta}_u)_{u\in[0,1]}$:  we define the measure $\tilde{\theta}_u$  as the derivative if $u\in L$, and zero otherwise.

We turn now to the claim: Let $\mathcal{X}$ be a countable set of functions which is dense in the space $\mathcal{C}_c(\R^d)$ of functions with compact support and let $\mathcal{X}_+$ be $\mathcal{X}\cup\{\rho\}$ where $\rho(x)=\|x\|_{\R^d}$. Let $L\subset [0,1]$ be the set such that at any time $u\in L$, $\theta_u$ is a probability measure and $u\mapsto \theta_{[0,u]}(f)$ has derivative $\theta_u(f)$ for any $f\in\mathcal{X}_+$ and note that $L$ has full mass. Then, as an increment $h$ goes to zero the measure $h^{-1}(\theta_{[0,u+h]}-\theta_{[0,u]})$ weakly converge to $\theta_u$ and as $\rho \in\mathcal{T}_1$ convergences holds also  in $\mathcal{T}_1$,  cf.\ \cite[Theorem 7.12]{Vi2}. Thus $L$ is a set of differentiation for any function with finite first moment.

\subsection{General description of the construction}\label{GeneDesc}

In what follows we shortly explain the canonical scheme to define $\hat{\pi}\in \hM(\hat{\mu},\nu)$ where the marginals $\hat{\mu}\in \Pi(\lambda,\mu)$ and $\nu\in \p$ are given. Recall that the resulting coupling $\pi=\pi_{[0,1]}=(\proj_{x,y})_{\#}\hat{\pi}$ of $\mu=\mu_{[0,1]}$ and $\nu$ fits more naturally to  the theory of optimal transportation than the lifted coupling $\hat{\pi}$.

Represent $\hat{\mu}$ in the form $(\mu_{[0,u]})_{u\in[0,1]}$. The first canonical operation, called \emph{shadow projection} on $\nu$, consists in building the curve $(\nu_{[0,u]})_{u\in[0,1]}$ from it (see Defintion \ref{def:shadow}). Hence the construction is complete if on a set $L\subset [0,1]$ of differentiation (of full measure) of $(\mu_{[0,u]})_u$ and $(\nu_{[0,u]})_u$ we know how to canonically choose a joint law $\pi_u$ of $\mu_u$ and $\nu_u$. In our situation, the martingale constraint and the fact that we used the shadow projection will make this choice uniquely determined and related to Kellerer's  \emph{hitting projection} (see Definition \ref{def:hitting}).  We will thus obtain $(\pi_u)_{u\in[0,1]}$ and equivalently $(\pi_{[0,u]})_{u\in[0,1]}$ and $\hat{\pi}$. This construction will be carried out in detail in the proof of Theorem \ref{thm:construction}.

\subsection{More details on Kellerer's approach}\label{sec:precisions}
In relation with Theorem \ref{Kellerer-Choquet} let us present in a more formal way the  Choquet representation given in  \cite{Ke72, Ke73}.

We denote by $\mathcal{F}(\R)$ the space of closed subsets of $\R$, and $\mathcal{I}$ the subspace of those elements $T\in\mathcal{F}(\R)$ such that $\sup T=-\inf T=+\infty$. The space $\mathcal{F}(\R)$ is endowed with the coarsest topology for which the mappings $T\in\mathcal{F}(\R)\mapsto d(x,T)\in [0,\infty]$ are continuous. By \cite[Satz 13]{Ke73} this topology is metrisable and compact, and by \cite[Satz 14]{Ke73} $\mathcal{I}$ is a $G_\delta$ subset of $\mathcal{F}(\R)$. 

\begin{defi}[Hitting projection of measure in/to a set]\label{def:hitting}
Let $T$ be an element of $\mathcal{I}$. For every $x\in T$, let $x_T^-=\sup(T\cap(-\infty,x])$ and $x_T^+=\inf(T\cap[x,+\infty))$. We define now the Kellerer dilation (\cite[Definition 16]{Ke73})
\begin{align*}P_T(x,\cdot)=
\begin{cases}
\delta_x&\text{if }x\in T;\\
(x_T^+-x_T^-)^{-1}[(x_T^+-x)\delta_{x_T^-}+(x-x_T^-)\delta_{x_T^+}]&\text{otherwise.}
\end{cases}
\end{align*}
Hence if $\mu\in \p$, the hitting projection of $\mu$ in $T$ is $\nu=\mu P_T$ and the hitting coupling of $\mu$ and $\nu$ 
is given by $\pi (A\times B) =\int_A P_T(x,B)\, \dd\mu(x)$; we shall abbreviate this by $\pi=\mu(\id\times P_T)$. 
\end{defi}
Note that if $T$ is not an element of $\mathcal{I}$ but $\spt(\mu)\subset [\inf T,\sup T]$, the hitting projection of $\mu$ still makes sense because $x^-_T$ and $x^+_T$ are finite and thus $\pi$ is a martingale transport plan. It is easy to replace $T$ by an element $T^*\in \mathcal{I}$ such that the resulting hitting couplings are equal, simply  define $T^*=]-\infty,\inf T] \cup \bar{T} \cup [\sup(T),+\infty[$. 
  Note that if $\mu\leqc\nu$, this remark applies to the set $T=\spt(\nu)$ since  $\spt(\mu)\subset [\inf T,\sup T]$ is satisfied.

\begin{pro}\label{unique}
Let $\mu$ be an element of $\p$ and $T\in \mathcal{F}(\R)$ satisfy $\spt (\mu)\subset [\inf T, \sup T]$. Let $\nu$ be $\mu P_T$. Then, the hitting coupling $\mu(\id\times P_T)$ is the unique element of $\M(\mu,\nu)$.
\end{pro}
\begin{proof}
This is \cite[Satz 25]{Ke73}. Alternatively, one may consider the decomposition of $(\mu,\nu)$ into irreducible components described in \cite[Theorem 8.4]{BeJu16}.
 This theorem specifies two canonical series $\mu=\mu\wedge \nu +\sum \mu_n$ and $\nu=\mu\wedge \nu + \sum \nu_n$ such that any $\pi\in \M(\mu,\nu)$ can be written as $\pi=(\id\times \id)_\#(\mu\wedge \nu)+\sum \pi_n$ with $\pi_n\in \M(\mu_n,\nu_n)$ for every $n$. This decomposition is based on the potential functions $u_\mu$ and $u_\nu$: the set $\{u_\mu< u_\nu\}$ is open and hence consists of a (finite or countable) union of open intervals $I_n$. Then $\mu_n= \mu|_{ I_n}$ whereas $\nu$ is concentrated on $\bar I_n$.  
In the present situation, $u_\nu$ is affine on each interval $\bar I_n$. This implies that necessarily $\pi_n=\mu_n(\id\times P_{\bar I_n})=\mu_n(\id\times P_T)$ and hence $\pi=\mu(\id \times P_T)$.
\end{proof}

 Let $\mathcal{U}$ be the space of probability measures $\chi$ on $\mathcal{I}$ such that $\int_\mathcal{I}d(0,T)\,\dd\chi(T)$ is finite, endowed with the coarsest topology such that $\chi\mapsto \int_\mathcal{I} h(T)\,\dd \chi(T)$ is continuous for every continuous functions $h:\mathcal{I}\mapsto \R$ with $\sup_T h(T)(1+d(0,T))^{-1}<+\infty$. \cite[Satz 18]{Ke73} asserts that $\mathcal{U}$ is metrisable and \cite[Satz 19]{Ke73} that $\mathcal{U}_0=\{\chi\in\mathcal{U}|\,\chi-\text{almost surely }T\supset T_0\}$ is compact. Moreover, according to \cite[Satz 20]{Ke73}, for $\mu\in \p$, the measure $\nu=\mu P_\chi$ has also finite first moments.

Kellerer establishes in \cite[Theorem 1]{Ke73}  that for a given $\mu\in \p$ the extreme points of $\{\nu\in \p:\mu\leqc\nu\}$ exactly are the measures $\mu P_T$. The latter set is not compact but going first through the compact and convex sets $\{\nu\in \p,\,\mu\leqc\nu\leqc \mu_{S}\}$ for sets $S\in \mathcal{I}$, Kellerer is able to derive a  Choquet representation in \cite[Theorem 4]{Ke73}. 

\begin{them}[A Choquet representation established by Kellerer]\label{KC}
Let $\mu$ and $\nu\in \p$ satisfy $\mu\leqc \nu$. Then there exist a probability measure $\chi\in\mathcal{U}$ such that
$$\nu=\int_\mathcal{I} \mu P_T\, \dd \chi(T)=:\mu P_\chi= \mu_\chi.$$
\end{them}

Let us give some precisions on this result and on Theorem \ref{Kellerer-Choquet}.
\begin{itemize}
\item Theorem \ref{Kellerer-Choquet} improves Theorem \ref{KC} with a uniqueness statement based on the natural order on $\mathcal{I}$. Our proof is independent  of Kellerer's. In particular we do not study what  the extreme elements of $E(\mu)=\{\omega\in \p|\,\mu\leqc \omega\}$ are.
\item As explained in Subsection \ref{KelSun} and using the notation of Subsection \ref{GeneDesc}, for every $u\in [0,1]$ the transport plan $\pi_u$ transfers $\mu_u=\mu$ onto $\nu_u$. We will see in Proposition \ref{central} that this transport is given through the \emph{hitting projection} of $\mu$ onto $\spt(\nu-\nu_{[0,u]})$. This description also provides a canonical choice for $R_u$ in Theorem \ref{Kellerer-Choquet} as well as a canonical barrier in Theorem \ref{MainTheorem} (2)  by setting $R=(]-\infty,0[\times \R)\cup\{(u,y)\in[0,1]\times\R,\, y\in \spt(\nu-\nu_{[0,u]})\}$.

\item In \cite{Ke73}, uniqueness of a measure on the extreme elements of $\{\omega\in\p,\,\mu\leqc \omega\}$ is not claimed and can easily be disproved. Set for instance $\mu=\delta_0$ and $\nu=\frac{\delta_{-2}+\delta_{-1}+\delta_1+\delta_2}{4}$. Taking the uniform probability measure on $T_1=\R\setminus ]-1,1[$ and $T_2=\R\setminus ]-2,2[$ on the one hand and an atomic measure on $\R\setminus ]-1,2[$, $\R\setminus ]-2,1[$ and $\R\setminus ]-2,2[$ we obtain two different representations of $\nu$. Another, more trivial, type of non-uniqueness  can be observed: In the previous example $T_1$ can also be replaced for instance by $\R\setminus (]-1,1[\cup]-20,-15[)$ providing the same measure on $\p$ but another measure on $\mathcal{I}$. 


\item From Theorem \ref{Kellerer-Choquet} it is not evident that $\chi$ satisfies the integrability condition appearing in the definition of $\mathcal{U}$. However according to Lemma 15 in \cite{Ke73}, this holds if and only if $\nu=\mu_\chi$ has finite first moments, i.e. is an element of $\p$. But this is one hypotheses in Theorem \ref{Kellerer-Choquet}. Hence, this theorem completely generalizes Theorem \ref{KC}, also with respect to $\chi\in \mathcal{U}$.
\end{itemize}



\subsection{Order relations and shadows}\label{sec:shadow}

On $\m$ we write $\mu\leqcp \nu$ if and only if there exists $\eta\in \m$ with $\mu\leqc \eta$ and $\eta\leqp \nu$. Here $\leqp$ means $\eta(A)\leq \nu(A)$ for every Borel set $A$. The order $\leqcp$ can also be characterized by asserting $\mu(f)\leq \nu(f)$ for every convex positive function $f$. We also introduce the stochastic order $\mu\leqs \nu$ that holds if $\mu(f)\leq \nu(f)$ for every integrable increasing function. This is equivalent to $G_\mu\leq G_\nu$ where $G_\mu$ denotes the unique increasing left-continuous function with $(G_\mu)_\#\lambda=\mu$, i.e.\ the quantile function. See \cite{Ju_ihp} for more details in the context of martingale optimal transport.

\begin{defi}[Definition of the shadow]\label{def:shadow}
If $\mu\leqcp\nu$ there exists a unique measure $\eta$ such that
\begin{itemize}
\item $\mu\leqc \eta$
\item $\eta\leqp \nu$
\item If $\eta'$ satisfies the two first conditions (i.e.\ $\mu\leqc \eta'\leqp \nu$), one has $\eta\leqc \eta'$.
\end{itemize}
This measure $\eta$ is called the \emph{shadow} of $\mu$ in $\nu$ and we denote it by $S^\nu(\mu)$.
\end{defi}
Shadows are sometimes difficult to determine. An important fact is that they have the smallest variance among the set of measures $\eta'$. Indeed, $\eta\leqc\eta'$ implies $\int x\, \dd\eta=\int x\, \dd\eta'$ and $\int x^2 \dd\eta\leq\int x^2 \dd\eta'$ with equality if and only if $\eta=\eta'$ or $\int x^2 \dd\eta=+\infty$. 
\begin{ex}[Shadow of an atom, Example 4.7 in \cite{BeJu16}]\label{un_atome}
Let $\delta$ be an atom of mass $\alpha$ at a point $x$. Assume that $\delta\leqe \nu$. Then $S^\nu(\delta)$ is the restriction of $\nu$ between two quantiles, more precisely it is
$\nu'=(G_\nu)_\#\lambda_{]s;s'[}$ where $s'- s =\alpha$ and the barycenter of $\nu'$ is $x$.
\end{ex}

The following result is one of the most important on the structure of shadows (Theorem 4.8 of \cite{BeJu16}).
\begin{pro}[Structure of shadows]\label{pro:shadowsum}
Let $\gamma_1,\gamma_2$ and $\nu$ be elements of $\m$ and assume that $\mu=\gamma_1+\gamma_2\leqcp\nu$. Then we have $\gamma_2\leqcp \nu-S^\nu(\gamma_1)$ and
$$S^\nu(\gamma_1+\gamma_2)=S^\nu(\gamma_1)+S^{\nu-S^\nu(\gamma_1)}(\gamma_2).$$
\end{pro}

An important consequence is that if $(\mu_{[0,u]})_u$ is a primitive curve and $\mu_{[0,1]}\leqc \nu$, then the curve $(\nu_{[0,u]})_u$ satisfies $\nu_{[0,u]}(\R)=u$ and using $\gamma_1=\mu_{[0,u]}$ and $\gamma_1+\gamma_2=\mu_{[0,v]}$ we obtain $\nu_{[0,u]}\leqp \nu_{[0,v]}$ for every $u\leq v$. Hence $(\nu_{[0,u]})_u$ is a primitive curve.


We consider the derivatives of shadow curves associated to a primitive curve.
\begin{pro}\label{central}
For $\hat\mu\in \Pi(\lambda,\mu)$ and $\mu\leqc\nu$, let $u\mapsto \mu_{[0,u]}$ have right derivative $\mu_{u_0}$ at $u_0$ and let $\nu_{[0,u]}$ be $S^\nu(\mu_{[0,u]})$. Then $(\nu_{[0,u]})$ has a right derivative at $u_0$. This derivative is given by $\mu_{u_0}P_T$ and $\spt (\mu_{u_0})\subset [\inf T,\sup T]$ where $T$ is the support of $\nu_{]u_0,1]}:=\nu-\nu_{[0,u_0]}$.
\end{pro}
Consider $h^{-1}(\nu_{[0,u_0+h]}-\nu_{[0,u_0]})=h^{-1}(S^\nu(\mu_{[0,u_0+h]})-S^\nu(\mu_{[0,u_0]}))=:\sigma_h$. According to Proposition \ref{pro:shadowsum} $\sigma_h$ equals
$$h^{-1}S^{\nu_{]u_0,1]}}(\mu_{]u_0,u_0+h]})$$
 where we set $\mu_{]u,v]}=\mu_{[0,v]}-\mu_{[0,u]}$. But we know that $h^{-1}(\mu_{[0,u_0+h]}-\mu_{[0,u_0]})=h^{-1}\mu_{]u_0,u_0+h]}$ tends to $\mu_{u_0}$ as $h\downarrow 0$. An easy scaling analysis shows that $\sigma_h$ can in fact be written
$$\sigma_h=S^{h^{-1}\nu_{]u_0,1]}}(h^{-1}\mu_{]u_0,u_0+h]}).$$ We are roughly speaking considering the shadow projection of $\mu_{u_0}$ into the infinite measure $\infty\cdot\nu_{]u_0,1]}$. As $\mathcal{T}_1$ is metric, it is enough for the convergence on the right of $u_0$ that we prove the following lemma:

\begin{lem}\label{shadows2Kellerer}
Let $\eta_n\to\eta$ in $\p$ and $H_n\to \infty$ and assume for every $n\geq 1$, $\eta_n\leqcp H_n\upsilon$. Then $S^{H_n\upsilon}(\eta_n)\to \eta P_T$ in $\p$ and $\spt(\eta) \subset [\inf T, \sup T]$ where $T=\spt(\upsilon)$.
\end{lem}
\begin{proof}

Note first that due to the convex order relation $\eta_n\leqcp \upsilon$ we have $\spt(\eta_n)\subset [\inf T, \sup T]$. Going to the limit $\spt(\eta)\subset [\inf T, \sup T]$ as well. We are left with the proof of $\eta_n\leqcp H_n\upsilon$.

1. We first prove the result if $\eta_n=\delta_x$ for every $n\in\N$. We prove in fact a little stronger statement: if $\gamma_n$ has mass less than or equal to one and $\gamma_n\leqp H_n\upsilon$, the sequence $S^{H_n\upsilon-\gamma_n}(\delta_x)$ converges to $\eta P_T$.  Moreover $x\in T^\circ$, $x\in T^c$ or $x\in \partial T$. In these three cases the result easily follow from Example \ref{un_atome}.

2. We assume now that for every $n\in\N$, we have $\eta_n=\eta=\sum^n_{k=1}a_k \delta_{x_k}$. The proof is an induction. The initial value $n=1$ has been previously done. We assume the statement for $n-1\geq 1$ and prove it for $n$ by using the decomposition $\eta=\eta'+a_n\delta_n$ where $\eta'=\sum^{n-1}_{k=1}a_k \delta_{x_k}$. We have \cite[Theorem 4.8]{BeJu16}
$$S^{H_n\upsilon}(\eta'+a_n\delta_{x_n})=S^{H_n\upsilon}(\eta')+S^{\beta_n}(a_n\delta_{x_n})$$
where $\beta_n=H_n\upsilon-S^{H_n\upsilon}(\eta')$. Each of the two terms converges to the Kellerer projection  of $\eta'$ resp.\ $a_n\delta_n$ onto $T$. Note that for the second projection we used the full strength of the statement proved in 1.

3. A general measure $\eta$ can be approximated using a convex combination of Dirac masses $\eta_k$ with $\eta_k \leqc \eta$ and such that $\eta_k\to \eta$ \cite[point 3. in proof of Proposition 2.34]{Ju_ihp}. We have
$$W(S^{H_n\upsilon}(\eta_k),S^{H_n\upsilon}(\eta))\leq W(\eta_k,\eta).$$
This goes to zero uniformly in $n$ as $k$ goes to infinity. But  $S^{H_n\upsilon}(\eta_k)\to \eta_k P_T$ as $n$ tends to infinity and the composition with $P_T$ is continuous (cf.\ \cite[Section 2.2]{Ke73}, this can be understood  easily from the action of $P_T$ on the potential functions). Hence we obtain the result for any constant sequence $\eta_n=\eta$.

4. If $\eta_n$ is a non-constant sequence 
$$W(S^{H_n\upsilon}(\eta_n),\eta P_T)\leq W(S^{H_n\upsilon}(\eta_n),S^{H_n\upsilon}(\eta))+W(S^{H_n\upsilon}(\eta),\eta P_T),$$ which tends to zero as required (\cite[Proposition 2.34]{Ju_ihp}). 
\end{proof}


\section{Construction}
Based on the preparations in the previous section we can now rigorously introduce the \emph{lefted shadow couplings}.
\begin{them}[Existence, construction, and uniqueness of the lifted shadow coupling]\label{thm:construction}
Let $\mu$ and $\nu$ be elements of $\p$ with $\mu\leqc \nu$. Let $\hat \mu$ be an element of $\Pi(\lambda,\mu)$. Then there exists a unique element $\hat\pi\in\hM(\hat\mu,\nu)$, the lifted shadow coupling of $\hat{\mu}$ and $\nu$, such that for every $u\in [0,1]$, the marginals of $\pi_{[0,u]}$ are $\mu_{[0,u]}$ and it shadow projection $S^\nu(\mu_{[0,u]})$. Moreover $\nu_{u}=\mu_{u}P_{T(u)}$ where $T(u)$ is the support of $\nu_{]u,1]}:=\nu_{[0,1]}-\nu_{[0,u]}$. 
\end{them}
Note that Theorem \ref{thm:construction} implies in particular that there exists a unique $\hat \pi$ satisfying Theorem \ref{MainTheorem} (4).
\begin{proof}[Proof of Theorem \ref{thm:construction}]
Let $\nu_{[0,u]}=S^\nu(\mu_{[0,u]})$ be as in the statement and let $(\mu_u)_{u\in[0,1]}$ and $(\nu_u)_{u\in[0,1]}$ be the derivative curves. According to Proposition \ref{central}, $\nu_u$ can for almost every $u\in [0,1]$ be identified with the Kellerer projection of $\mu_u$ in $T(u):=\spt(\nu_{]u,1]})$ where $\nu_{]u,1]}=\nu-\nu_{]0,u]}$. Hence we can define $\pi_u=\mu_u(\id\times P_{T(u)})\in \M(\mu_u,\nu_u)$ for almost every $u$, and then the corresponding $\hat{\pi}$ and $\pi_{[0,u]}=\int_0^u \pi_{t}\, \dd t$. 

Conversely, let $\hat\pi$ and $(\pi_{[0,u]})_{u\in[0,1]}$ be as in the statement. The curve has marginals $\mu_{[0,u]}$ and $\nu_{[0,u]}=S^\nu(\mu_{[0,u]})$ so that one can apply Proposition \ref{central}. At points $u$ where the derivatives $\pi_u$, $\mu_u$ and $\nu_u$ exist, we have $\pi_u\in\M(\mu_u,\nu_u)$ and $\nu_{[0,u]}=\mu_u P_{T(u)}$ as in the last paragraph, proving the uniqueness part of the statement.
\end{proof}

\subsection{Examples of shadow couplings}
We now further discuss canonical classes of families $(\mu_{]0,u]})_{u\in ]0,1]}$ giving rise to particular shadow couplings. 
We consider different canonical lifts $\hat{\mu}\in\Pi(\Lg,\mu)$ of $\mu$. The first one is the monotone coupling of $\Lg$ and $\mu$ and the second the independent coupling $\Lg\times \mu$.
\begin{itemize}
\item $\mu_{[0,u]}=(G_\mu)_\#\Lg|_{[0,u]}$ (corresponding to the left-curtain coupling);
\item $\mu_{[0,u]}=u\cdot\mu$ (corresponding to the sunset coupling);
\item $\mu_{[0,u]}=S^\mu(u.\delta_m)$ where $m=\int x\,\dd \mu(x)$ (corresponding to the middle curtain coupling). Recall Example \ref{un_atome} for the shadow of an atom.
\end{itemize}
Figures \ref{un}, \ref{deux} and \ref{trois} illustrate the corresponding Skorokhod embeddings when $\mu$ and $\nu$ are two uniform laws.

\subsubsection{The left- and right-curtain couplings}
This case corresponds to the construction given in \cite{BeJu16}, even though the construction described there appears slightly different.
 In fact for $u=F_\mu(x)$ the three marginals of $\hat\pi$ are $\lambda|_{[0,u]}$, $\mu|_{[0,x]}$ and $S^\nu(\mu_{[0,x]})$ so that for every $x\in \R$, $\pi$ has marginals $\mu|_{[0,x]}$ and $S^\nu(\mu_{[0,x]})$. In \cite{BeJu16} this was used to define the left-curtain coupling $\pi_\lc$. 

In an entirely symmetric fashion we can define the right-curtain coupling through $\mu_{[0,u]}=(G_\mu)_\#\Lg|_{[1-u,1]}$ for $u\in[0,1]$.

\begin{figure}[ht]
\begin{center}
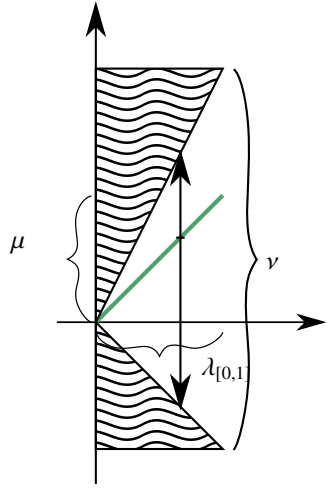
\caption{Left-curtain coupling of uniform measures}\label{un}

\end{center}
\end{figure}


\subsubsection{The sunset coupling}\label{vertic} 
In this case we have $\mu_u=\mu$ for almost every $u\in[0,1]$ and $\nu_u=\mu P_{T(u)}$ where $T(t)=\spt(\nu-S^\nu(u\cdot\mu))$. Hence
$$\textstyle \nu=\int_0^1\mu P_{T(ut)}\,\dd u$$
Of course $T(u)$ can be replaced by $T^*(u)=T(u)\cup(]-\infty,\inf \nu]\cup [\sup \nu,+\infty[)\in\mathcal{I}$. This is a possibility  to make not only the almost-everywhere defined family $\nu_u=\mu_u P_{T(u)}$  unique as in Theorem \ref{Kellerer-Choquet} but also the measure $\chi$.

\begin{figure}[ht]\begin{center}
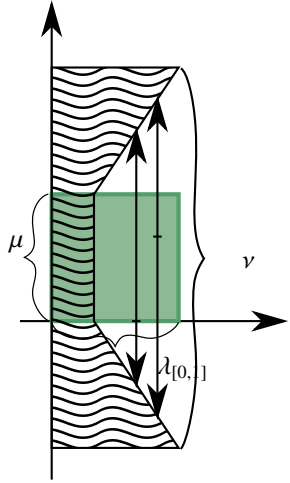
\caption{Sunset coupling of uniform measures}\label{deux}

\end{center}
\end{figure}

\subsubsection{The middle-curtain coupling} As variant of the (left-) curtain coupling, we introduce a 
 middle-curtain coupling. Under the condition that  $\mu$ and $\nu$ are in diatomic order ($\mu\preceq_{DC} \nu$), i.e.
\begin{align}\label{condition}
S^\mu(u\delta_m)\leqc S^\nu(u\delta_m)\quad\text{for every }u\leq 1,
\end{align}
where $m$ denotes the center of $\mu$, 
  the middle-curtain coupling coincide with an exceptionaly simple martingale transport plan that has been introduced in \cite[Section 4]{Ju_seminaire}. (Clearly, the diatomic order is more restrictive than the convex order.)


We define the middle-curtain coupling as the shadow coupling corresponding to the family $\mu_{[0,u]}=
S^\mu(u\delta_m), u\in [0,1]$.
Note that this corresponds to $\mu_u=a(u)\delta_{f(u)}+b(u)\delta_{g(u)}$ where
\begin{itemize}
\item $a(u)f(u)+b(u)g(u)=m$;
\item $f$ is decreasing and $g$ is increasing.
\end{itemize}
If $\mu\preceq_{DC} \nu$, it is straightforward to establish that  $\nu_{[0,u]}=S^\nu(\mu_{[0,u]})=S^\nu(u\delta_m)$, so that $(\nu_u)_{u\in [0,1]}$ is of the same type as $(\mu_u)_{u\in [0,1]}$, that is
$$\nu_u=a'(u)\delta_{f'(u)}+b'(u)\delta_{g'(u)}.$$
Moreover $f'\leq f$ and $g\leq g'$ and $\pi_u$ is concentrated on the four oriented pairs $(f(u),f'(u))$, $(g(u),g'(u))$, $(f(u),g'(u))$ and $(f'(u),g(u))$. More explicitely.
\begin{align*}
\pi_u=&\frac1{g'(u)-f'(u)}\left([(g'(u)-f(u))\delta_{f(u),f'(u)}+(f(u)-f'(u))\delta_{f(u),g'(u)}]\right.\\
&\left.+[(g'(u)-g(u))\delta_{g(u),f'(u)}+(g(u)-f'(u))\delta_{g(u),g'(u)}]\right).
\end{align*}
Finally note that if $(\mu_t)_{t\in T}$ is a family of probability measures indexed by a partial order $T$ so that $s\leq t$ implies $\mu_s\preceq_{DC} \mu_t$, then there exist a martingale $(X_t)_{t\in T}$ with $\law(X_t)=\mu_t$ for every $t\in T$, \cite[Theorem 4]{Ju_seminaire}.
\begin{figure}[ht]
\begin{center}
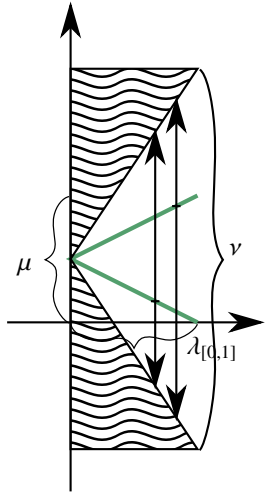
\caption{Middle-curtain coupling of uniform measures}\label{trois}
\end{center}
\end{figure}

\subsubsection{Comparison with the stochastic order, the quantile and independent couplings.}

The stochastic order and the convex order share several common features and a parallel presentation can be given for  shadow couplings of measures in convex order and the classical couplings. As already explained in \cite{BeJu16} the left-curtain coupling can be considerd as the natural counterpart of the quantile coupling. Let us see that in the construction of the left-curtain coupling, replacing the shadows by \emph{stochastic shadows} we obtain the quantile coupling. We assume $\mu\leqs \nu$. Let the stochastic shadow of a (sub)measure $\mu'\leqp \mu$ in $\nu$ be the smallest measure $\eta$ in $\leqs$ such that $\mu'\leqs \eta$ and $\eta\leqp \nu$. The lift $\hat{\mu}$ is the quantile coupling of $\lambda$ and $\mu$, so that $\mu_{[0,u]}=(G_\mu)_\#\lambda|_{[0,u]}$. For $\mu\leqs \nu$, the stochastic shadow of this measure is simply $(G_\nu)_\#\lambda|_{[0,u]}$ independently of $\mu$. Therefore, for the curves $(\mu_{[0,u]})_u$ and $(\nu_{[0,u]})$ defined in this way we recognize the quantile coupling.

We relax now the first condition $\mu_{[0,u]}\leqs \eta$ for defining the stochastic shadow (it is always satisfied for measures of mass $u$ with $\eta\leqp \nu$), so that we can establish a coupling not only in the case $\mu\leqs\nu$ but in general,
keeping $\nu_{[0,u]}=(G_\nu)_\#\Lg|_{[0,u]}$. We obtain again the quantile coupling. 
The left-curtain and the quantile coupling are also analogous on the level optimality properties, see \cite[Sections 1.2, 1.3]{BeJu16}.

While the left-curtain coupling can be viewed as  the quantile coupling of the convex order world, we will explain next in which sense the sunset coupling corresponds to the independent (aka  product) coupling. As before we do not assume $\mu\leqs \nu$ and we still consider stochastic shadows given through $\nu_{[0,u]}=(G_\nu)_\#\Lg|_{[0,u]}$. However we take the same lift $\hat \mu=\lambda\times \mu$ as for the sunset coupling, so that $\mu_u=\mu$. It is then easy to identify the derivative in the target space as $\nu_u=\delta_{G_\nu(u)}$. Therefore the kernel final coupling writes
$$\pi_{[0,1]}=\pi=\int_0^1 \mu(\id \times P_{G_\nu(u)})\,\dd u.$$
This is nothing but $\mu\times \nu$.
Apart from the Kellerer-Choquet representation given  in Paragraph \ref{vertic}, we have encountered another sense in which the sunset coupling is particularly canonical: It is the product coupling of the convex order world.

\section{Proof of Theorem \ref{MainTheorem}}
Throughout this section we assume that $\mu, \nu$ are in convex order and that $\hat \mu\in \Pi(\lambda,\mu)$. Given a measurable $c: \spac\to \R_+$ we consider the optimization problem \begin{equation}\label{G1} \textstyle
P:=P_c:= \inf\left\{ \int c\,  \dd\hat\pi:\hat\pi\in \hM(\hat\mu,\nu)\right\}.
\end{equation}
\begin{proposition}\label{4ToOpt}
Assume that $\hat \pi$ is the lifted shadow coupling corresponding to a curve $(\pi_{[0,u]})_{u\in [0,1]}$ as in Theorem \ref{MainTheorem} (4). Then for all $p\in [0,1], q\in \R$, $\hat \pi$  is an optimizer of \eqref{G1} for $c_{p,q}(u,x,y)=\I_{u\leq p} |y-q|$. 
\end{proposition}
\begin{proof} Let $\hat{\pi}\in\hM(\hat{\mu},\nu)$ be a shadow coupling and $(p,q)$ as in the statement. Then $\int c_{p,q}\,\dd\hat{\pi}=\int |y-q|\,\dd\nu_{[0,p]}(y)$ where we recall $\nu_{[0,p]}=S^\nu(\mu_{[0,p]})$. More generally if $\hat{\gamma}$ is an element of $ \hM(\hat{\mu},\nu)$ we have $\int c_{p,q}\dd\hat\gamma=\int |y-q|\,\dd\beta^p(y)$ where $\mu_{[0,p]}\leqc \beta^p$ and $\beta^p\leqp \nu$ (in fact $\beta^p:=(\proj_y)_\#\gamma|_{[0,p]\times \R\times \R})$). Therefore $\nu_{[0,q]}\leqc \beta^p$ and as $y\mapsto |y-q|$ is convex, we have proved that $\hat{\pi}$ is a minimizer.
\end{proof}
Actually if $\hat{\gamma}$ is a minimizer, for any $p\in[0,1]$ the measures $\beta^p$ and $\nu_{[0,p]}$ have the same potential function. Thus they are equal and the curve $p\mapsto (\proj_y)_\#\gamma|_{[0,p]\times \R\times \R})$ is completely determined. But we have proved the uniqueness of such couplings in Theorem \ref{MainTheorem}. Hence $\hat{\gamma}=\hat{\pi}$.

Recall from Theorem \ref{MainTheorem} that a set $\hat \Gamma\subseteq \spac$ is called \emph{monotone}
if for all $u,v,x,x',y^-,y^+, y'$ such that $ u<v, (u,x,y^-), (u,x,y^+), (v,x',y')\in \Gamma$ it holds $y'\notin ]y^-,y^+[.$

\begin{proposition}\label{OptToMon}
Assume that $\hat \pi\in \M(\mu, \nu)$ satisfies one of the following assumptions:
\begin{enumerate}
\item For all $u\in [0,1], q\in \R$, $\hat \pi$  is an optimizer of \eqref{G1} for $c_{p,q}(u,x,y)= \I_{u\leq p} | y-q|$. 
\item $\hat \pi$ is an optimizer of \eqref{G1} for $c(u,x,y)=(1-u) \sqrt{1+y^2}$.
\end{enumerate}
Then there is a monotone set $\hat \Gamma$ such that $\hat\pi(\hat \Gamma)=1$. 
\end{proposition}
\begin{proof}
We will establish the assertion under the first assumption, the argument based on the second assumption is very similar. 

Using the notation and the monotoneity principle from \cite{BeGr14}, we pick for each $(u,q)\in ([0,1]\times \R) \cap \Q^2$ a monotoneity set $\Gamma_{(u,q)}$ for the cost function and set $$ \Gamma:= \bigcap_{(u,q)\in [0,1]\times \R \cap \Q^2} \Gamma_{(u,q)}.$$
Assume for contradiction that there exist 
$s,t,x,x',y^-,y^+, y'$ such that 
$$ s< t, a:=(s,x,y^-), b:=(s,x,y^+), c:=(t,x',y,)\in \Gamma \Rightarrow y'\in ]y^-,y^+[.$$ 
Pick $\lambda$ such that $y'= (1-\lambda) y^- + \lambda y^+$, $u\in ]s,t[$, and $q$ very close to $y'$ (in comparison to $y^-,y^+$). Set $$ a':=(t,x',y^-), b:=(t,x',y^+), c:=(s,x,y,).$$
Then $$ \alpha:= (1-\lambda) \delta_a + \lambda\delta_b+ \delta_c, \alpha':=(1-\lambda) \delta_{a'} + \lambda\delta_{b'}+ \delta_{c'}$$ are competitors with $\supp \alpha \subset \hat \Gamma$ and $\int c_{u,q}\, \dd\alpha > \int c_{u,q}\, \dd\alpha'$, contradiction.   
\end{proof}

In the next result we establish that any $\hat \pi$  which is monotone admits a barrier representation as in Theorem \ref{MainTheorem} (2). 
\begin{proposition}\label{Monotone2Barrier}
Let $\hat\pi \in \M(\hat \mu, \nu)$ be a transport plan concentrated on a monotone set $\hat \Gamma $. 
Define barriers 
\begin{align}
R_o\, &:=\{(s,y)\in [0,1]\times \R: \exists t> s, (t,y)\in \hat \Gamma\}\\
R_c\, &:=\{(s,y)\in [0,1]\times \R: \exists t\geq s, (t,y)\in \hat \Gamma\}.
\end{align}
Consider a process $(Z_t)_{t\geq 0}=(Z^1_t, Z^2_t)_{t\geq 0}=(Z^1_0, Z^2_t)_{t\geq 0}$ on some probability space  which takes values in  $[0,1]\times \R$ and is specified through
\begin{enumerate}
\item $Z_0\sim \hat\mu$,
\item $Z_t=Z_0+(0,B_t)$, where $(B_t)_t$ is (one dimensional) Brownian motion.
\end{enumerate}
and write $\tau_o, \tau_c$ for the first time $Z$ hits $R_o$ resp.\ $R_c$. Then $\tau_o=\tau_c$ a.s.\ and  
\begin{align}\label{CorrectMT}(Z_0, Z_{\tau_o})\sim (Z_0, Z_{\tau_c})\sim \hat \pi.\end{align}
The martingales $t\mapsto Z_{t\wedge \tau}$ and $t\mapsto B_{t\wedge \tau}$ are uniformly integrable.

There exist Borel maps $T_{\text{up}}, T_{\text{down}}: [0,1]\times \R\to \R, T_{\text{down}}(x)\leq x\leq T_{\text{up}}(x)$ such that
$$\hat\pi\{(u,x,T_i(x)): i\in \{\text{up},\text{down}\}, (u,x) \in   \spa\}=1.$$ 
\end{proposition}
\begin{proof}
Fix a disintegration $(\pi_{u,x})$ of $\hat \pi$ wrt $\hat \mu$ and write $ \Gamma_{u,x}$ for the section of $\hat \Gamma $ in $(u,x)$. Then $\hat\mu(\Gamma_0)=1$, where
$$\textstyle \Gamma_0=\left\{(u,x): \pi_{u,x} (\Gamma_{u,x})=1, \int |y|\, \dd\pi_{u,x}<\infty,  \int y\, \dd\pi_{u,x}= x\right\}.$$
Define for each $(u,x)\in \spa$, $\tau_{u,x}$ to be the Azema-Yor solution (say) of the Skorokhod embedding problem such that $B_{\tau_{u,x}}\sim\pi_{u,x}$. Then define a stopping time $\tau$ such that conditionally on $Z_0=(u,x)$ we have $\tau=\tau_{u,x}$. It follows that 
$(Z_0,Z^2_\tau)\sim \hat \pi$ and that for all elements $\omega$ of a full measure set $\Omega_0$ we have $(Z^1_0(\omega),Z^2_0(\omega),Z^2_\tau(\omega))\in \hat\Gamma$.
Next we claim that there exists a full measure subset $\Omega_1$ of $\Omega_0$ such that for all $\omega \in\Omega_1$ and every $t<\tau(\omega)$, the following assertion holds true:

{\it Continuation Assertion on $(\omega,t)$.} There are $\omega_i\in \Omega_0,  i=1,2$ satisfying
\begin{enumerate}
\item $t<\tau(\omega_i)$, $(Z_{s}(\omega))_{s\leq t}=(Z_{s}(\omega_i))_{s\leq t}$ for $i=1,2$,
\item  $Z^2_{\tau}(\omega_1)<Z^2_t(\omega) < Z^2_{\tau}(\omega_2)$. 
\end{enumerate}
Assume for contradiction that  the set $$\{(\omega, t): t<\tau(\omega)\mbox{ and Continuation Assertion fails}\}=:D$$
is not evanescent, i.e.\ that $\proj_{\Omega}(D)$ does not have $\P$-measure $0$. 
Set 
\begin{align*}
 D^-&\,:=\{(\omega, t)\in  \ll0,\tau\ll\,: \omega_1\in \Omega_0,
t<\tau(\omega_1), (Z_{s}(\omega_1))_{s\leq t}=(Z_{s}(\omega))_{s\leq t} \Rightarrow Z_\tau(\omega_1)\leq Z_t(\omega)\}\\
D^+&\,:=\{(\omega, t)\in \ll0,\tau\ll\,: \omega_2\in \Omega_0,
r<\tau(\omega_2), (Z_{s}(\omega_2))_{s\leq t}=(Z_{s}(\omega_2))_{s\leq t} \Rightarrow Z_\tau(\omega_2)\geq Z_t(\omega)\}
\end{align*}
such that $D=D^-\cup D^+$. 
If $D$ is not evanescent, then by the optional section theorem there exists a stopping time $\sigma$ such that $\P(\sigma < \infty)> 0 $ and $$\{(\omega, \sigma(\omega )): \sigma(\omega)<\infty\}\subseteq D^-\text{ or } \{(\omega, \sigma(\omega )): \sigma(\omega)<\infty\}\subseteq D^+.$$ Combined with the strong Markov property this leads to a contradiction with the optional stopping theorem.

We claim that on $\Omega_1$
\begin{align}
\tau_c\leq \tau \leq \tau_o.
\end{align}
Note that the first inequality is satisfied by definition of $\tau_c$. To establish the second inequality we assume for contradiction that there exists $\omega \in \Omega_1$ such that $\tau_o(\omega) < \tau(\omega)$. 

Then $t^*:=\min\{t\geq 0: Z_t(\omega) \in R_o\} < \tau(\omega)$.
Set $y':=Z^2_{t^*}(\omega)$ and
 $(u,x)=(Z^1_0(\omega),Z^2_{0}(\omega))$.
By definition of $R_o$, there  exist $v>u$ and $x'$ such that $(v,x',y')\in \hat\Gamma$.
Pick $\omega_i, i=1,2$ according to the Continuation Assertion.
Setting $y_i=Z^2_{\tau}(\omega_i), i=1,2$, we have $(u,x,y_i)\in \hat\Gamma$, contradiction.

By Lemma \ref{BarrierEquality} $\tau_c=\tau_o$ almost surely hence \eqref{CorrectMT} holds. 

To see that $(Z_{t\wedge \tau})$ (resp.\ $(B_{t\wedge \tau})$) is uniformly integrable we recall a result of Monroe~\cite{Mo72} which asserts that a solution $\tau$ of the Skorokhod problem is minimal (i.e.\ there is no strictly smaller solution) if and only if Brownian motion up to time $\tau$ is uniformly integrable. In the present context it is straight forward to verify that $\tau$ provides a minimal embedding of $\nu$ wrt $Z^2$ (we refer to \cite[Proposition 4.1]{BeHeTo15} for complete details), hence $(Z^2_{t\wedge \tau})$ is uniformly integrable.

The rest is immediate.
\end{proof}

In the proof we used the following lemma from \cite{BeCoHu16} (we include the proof for the convenience of the reader).
\begin{lemma}\label{BarrierEquality}
Let $\hat\mu$ be a probability measure on $\R^2$ such that the projection onto the horizontal axis $\proj_x \hat\mu$ is continuous (in the sense of not having atoms) and let $\phi: \R\to \R$ be a Borel function. Set 
$$R_\op:= \{(x,y): x>\phi(y)\},\quad R_\cl:= \{(x,y): x\geq\phi(y)\}.$$
Start a vertically moving Brownian motion in $\mu$ and define
$$ \tau_\op:= \inf\{ t: (x, y+B_t)\in R_\op\}, \quad \tau_\cl:= \inf\{ t: (x, y+B_t)\in R_\cl\}.$$
Then $\tau_\cl=\tau_\op$ almost surely.
\end{lemma}
\begin{proof}
 Obviously $\tau_\cl\leq \tau_\op$. 

 We say that $y$ is a local minimum of $\phi$ if $\phi(y')\geq \phi(y)$ for all $y'$ in a neighborhood of  $y$. Set 
$$I:=\{\phi (y): y \text{ is a local minimum of } \phi\}.$$
It is then not difficult to prove (and certainly well known) that $I$ is at most countable: assume by contradiction that there exist an uncountably family $A\subseteq \R$ and corresponding neighborhoods $(a-\eps_{sun}, a+\eps_a), a\in A$ such that $\phi(x)\geq \phi(a)$ for $x\in (a-\eps_a, a+\eps_a)$ and $a\neq a'$ implies $f(a)\neq f(a')$. Passing to an uncountable subset of $A$, we can assume that there is some $\eta >0$ such that $\eps_a>\eta$ for all $a\in A$. For $a\neq a'$ we cannot have $|a-a'| < \eta$ for then  $a\in (a'-\eps_{a'}, a'+\eps_{a'})$ as well as $a'\in (a-\eps_{a}, a+\eps_{a})$ which would imply that $f(a)=f(a')$. Hence $|a-a'|\geq \eta$ which implies that $A$ is countable, giving a contradiction. 

\smallskip

On the complement of $I\times \R$ we have almost surely
\begin{align}\tau_\op=0 \quad \Longleftrightarrow \quad \tau_\cl=0\end{align}
as a consequence of the strong Markov property.\end{proof}

We have thus obtained an interpretation of monotone transport plans in terms of a barrier-type solution to the Skorokhod problem.  This interpretation is useful for us since  it allows us to use a short argument of Loynes \cite{Lo70} (which in turn builds on Root \cite{Ro66}) to show that there is only one monotone transference plan.

\begin{lemma}[cf.~Loynes \cite{Lo70}]\label{BarrierUniqueness}
Let $\hat\pi_1$, $\hat\pi_2$ be monotone transport plans in $\hM(\hat\mu,\nu)$, with corresponding maps $T^\P=(T^i_{\text{up}}, T^i_{\text{down}})$  and denote by $R^{{\hat \pi_i}}, i=1,2$ the corresponding `closed' barriers as in Proposition \ref{Monotone2Barrier}. Then $\tau_{R^{{\hat \pi_1}}}= \tau_{R^{{\hat \pi_2}}}$, a.s.
\end{lemma} 

\begin{proof}
For a set $A\subset\R$, we abbreviate $R_i(A):= R^{{\hat \pi_i}}\ \cap\ (\R\times A)$ and $\tau_i=\tau_{R^{{\hat \pi_i}}}$ for $i=1,2$. Denote
 \begin{align}
 K:= \big\{y: m_1(y) > m_2(y) \big\}   
 &\mbox{ where}&
 m_i(y):=\sup \{m: (m,y) \in R^{{\hat \pi_i}}\},~~i=1,2.
 \end{align}
Fix a trajectory $(Z_t)_t=(Z_t(\omega))_t$ such that $Z^2_{\tau_2} \in K$. Then $(Z_t)_t $ hits ${R_2}(K)$ before  it enters $R_2(K^C)$. 
But then $(Z)_t $ also hits $ R_1(K)$ before it enters $R_1(K^C)$. Hence 
$$
B_{\tau_2} \in K \quad \Longrightarrow \quad B_{\tau_1} \in K.
$$ 
As both stopping times embed the same measure, this implication  is an equivalence almost surely, and we may set $\Omega_K:=\{B_{\tau_1} \in K\}= \{B_{\tau_2} \in K\}$. On $\Omega_K$ we have $\tau_1 \leq \tau_2$ while $\tau_1 \geq \tau_2$ on $\Omega_K^C$. Then, for all Borel subset $A\subset\R$:
 \begin{align}
 \P\big[B_{\tau_1 \wedge \tau_2}\in A\big]
 &=&
 \P\big[B_{\tau_1 \wedge \tau_2}\in A, \Omega_K\big]
 +\P\big[B_{\tau_1 \wedge \tau_2}\in A, \Omega_K^c\big]
 \\
 &=&
 \P\big[B_{\tau_1}\in A, \Omega_K\big]
 +\P\big[B_{\tau_2}\in A, \Omega_K^c\big]
 \\
 &=&
 \P\big[B_{\tau_1}\in A\cap K\big]
 +\P\big[B_{\tau_2}\in A\cap K^c\big]
\\
&=&
 \P\big[B_{\tau_2}\in A\cap K\big]
 +\P\big[B_{\tau_2}\in A\cap K^c\big]
\\
&=&
 \P\big[B_{\tau_2}\in A\big]
\end{align}
since $B_{\tau_i}\sim\nu$. Hence $\tau_1 \wedge \tau_2$ embeds $\nu$. Similarly, we see that  $\tau_1 \vee \tau_2$ also embeds $\nu$. Since $\tau_1$ and $\tau_2$ are both minimal embeddings, we deduce that $\tau_1 \wedge \tau_2=\tau_1 $ as well as $\tau_1 \wedge \tau_2=\tau_2. $\end{proof}

Taking the results of this section we can now establish our main theorem.
\begin{proof}[Proof of Theorem \ref{MainTheorem}]
We have already seen in Theorem \ref{thm:construction} that there exists  $\hat\pi\in \hM(\hat\mu,\nu)$ satisfying Theorem \ref{MainTheorem} (4). By virtue of Propositions \ref{4ToOpt}, \ref{OptToMon} we have that $\hat \pi$ is monotone as required in \ref{MainTheorem} (3) and by Proposition \ref{Monotone2Barrier} $\hat \pi$ admits a barrier type representation as in \ref{MainTheorem} (2). Moreover, by Lemma \ref{BarrierEquality} we find that there exists a unique such $\hat\pi$.

Finally, by the standard compactness-continuity argument  there exists $\hat\pi$ which solves the optimization problem in Theorem \ref{MainTheorem} (1), by Proposition \ref{OptToMon} it is monotone and hence uniquely determined as before.
\end{proof}

Note that in the proof of Theorem \ref{MainTheorem}, we did not use the uniqueness part in the statement of Theorem \ref{thm:construction}; rather we have obtained a second derivation of this uniqueness property based on Lemma \ref{BarrierUniqueness}. 

We close this section with a  remark on the implication of the above results for the curtain coupling.
\begin{remark}\label{FamousCoro} We consider the curtain coupling $\pi_{\lc}$ corresponding to the case where the lift $\hat \mu$ is given by the monotone rearrangement between Lebesgue measure and $\mu$. 
Assume for simplicity that $\mu$ has no atoms such that the 
lift $\hat\mu$ is concentrated on the graph of a 1-1 function  elements of $\M(\mu,\nu)$ correspond in a 1-1 manner to elements of $\hM(\hat\mu,\nu)$. It then follows from the respective optimality property of $\hat \pi$ that $\pi_{\lc} $ minimizes 
\begin{align}
\gamma \mapsto \int \phi(x)\psi(y)\, \dd\gamma(x,y)
\end{align}
on the set $\M(\mu,\nu)$, where $\varphi\geq 0$ is strictly decreasing and $\psi\geq 0$ is strictly convex and the minimum over $\M(\mu, \nu)$ is finite. 
Moreover, there exist a Borel set $S\subseteq \R$ and two measurable  functions $T_1, T_2:S\to \R$ such that 
\begin{enumerate}
\item $\pi_\lc$ is concentrated on the graphs of $T_1$ and $T_2$.
\item For all $x\in\R,\, T_1(x)\leq x \leq T_2(x)$.
\item For all $x<x'\in \R$, $T_2(x) < T_2(x')$ and $T_1(x')\notin \,]T_1(x), T_2(x)[.$
\end{enumerate}
This recovers \cite[Corollary 1.6]{BeJu16}.
\end{remark}

%
%

\section{The sunset coupling as a non-optimizer and shadow couplings as optimizers to general transport problems.}

An important message of \cite{BeJu16, HeTo13} is that the left-curtain couplings are characterized as the optimizers to martingale optimal transport problems for a large class of cost functions. This goes together with the fact that the  support of the left-curtain coupling is typically a very `small' set -- if $\mu$ is continuous, it is contained in the graphs of two functions. 

In contrast, the sunset coupling typically has a `large' support. Hence we do not expect it to solve  a martingale transport problem except in trivial instances. This is underlined by the following simple example.
\begin{example}\label{NonOptimizer}
Let $\mu$, $\nu$ be measures in convex order such that 
\begin{enumerate}
\item $\text{conv}(\supp(\mu))\cap \supp(\nu)=\emptyset$
\item $\mu, \nu$ consist of finitely many atoms.
\end{enumerate}
Assume $c$ is such that the sunset coupling is optimal for the martingale  transport problem. Then \emph{all} elements of $ \M(\mu, \nu)$ are optimal for the martingale  transport problem.  
\end{example}
\begin{proof}We first note that $\supp (\pi_{\sun})= \supp(\mu)\times \supp(\nu)$ under our assumptions on $\mu, \nu$.

 In the present atomic case, the martingale transport problem can be formulated as a linear programming problem and it admits a natural dual problem for which strong duality holds. It follows from this that every martingale transport plan $\pi\in \M(\mu, \nu)$ satisfying $\supp(\pi) \subseteq \supp (\pi_{\sun}) $ is optimal. 
\end{proof}
For simplicity, we have stated Example \ref{NonOptimizer} for discrete marginals but (with some work) it is not difficult to see that the same phenomenon carries over to more general cases. 

\medskip

However we find it interesting to note that shadow couplings posses optimality properties in a different sense:
Recently, Gozlan et al. \cite{GoRoSaTe14} introduced a framework for \emph{general} transport problems. As in the classical case one optimizes over the set of transport plans $\pi\in \P(\mu, \nu)$, where $\mu, \nu$ are probabilities on Polish spaces $X,Y$. In contrast to the classical case, more general cost functionals are considered. Writing ${\mathcal P}(Y)$ for the set of all probability measures on $Y$, a general cost is a function $C:X\times {\mathcal P}(Y)\to [0,\infty]$ and its associated transport costs are 
\begin{align}
T_C(\nu|\mu):= \inf_{\pi\in \Pi(\mu, \nu)}\int C(x,\bar\pi_x)\, \dd\mu(x),
\end{align}
where we use $(\bar\pi_x)_x$ to denote disintegration wrt\ $\mu$. 

The shadow couplings appear as optimizers to such general transport problems. E.g.\, we will see below that the sunset coupling is the unique optimizer
for the general transport cost function
\begin{align}\label{GenVar}
C(x, \bar \pi):=
\inf_{\alpha\in \hM(\lambda \times \delta_x, \bar \pi)}\int (1-u) \sqrt{1+y^2} \, \dd\alpha(u,x', y).
\end{align}
Here the function $(u,y)\mapsto (1-u)\sqrt{1+y^2}$ could be replaced by any function of the form $(u,y)\mapsto \phi(u)\psi(y)$, where $\phi$ is strictly increasing and $\psi$ strictly convex and sufficiently integrable wrt the given marginals.
The cost function defined in \eqref{GenVar} exhibits a relatively intuitive behavior: If $\bar \pi$ does not have  center $x$ the costs equal $+\infty$. If $\bar \pi$ is centered around $x$, the more $\bar \pi$ is spread out, the higher are the costs.
 
 \begin{proposition}
Let $\mu$ and $\nu$ be in convex order and fix a disintegration $(\bar \mu_x)_x$ of $\hat \mu$ wrt\ $\mu$ and set 
\begin{align}\label{GenVar2}
 C^{\hat \mu}(x, \bar \pi):=\inf_{\alpha\in \hM(\bar\mu_x\times \delta_x,\bar \pi)} \int  (1-u)\sqrt{1+y^2}\, \dd\alpha(u, x', y).
\end{align}
Then the shadow coupling associated to $\hat\mu$ is the unique optimizer of the general transport problem associated to $ C^{\hat \mu}$. 
\end{proposition}
We note that the solution to the optimization problems  \eqref{GenVar} / \eqref{GenVar2} is straightforward to characterize in the nontrivial case where $x$ is the barycenter of $\bar \pi$:
The optimizer $\alpha$ is the unique element of $\hM(\bar\mu_x\times \delta_x, \bar \pi)$ which is concentrated on the graphs of two functions $T_{\text{up}}:[0,1]\to [x, \infty)$, $T_{\text{down}}:[0,1]\to (-\infty, x]$, where $T_{\text{up}}$ is increasing and $T_{\text{down}}$ is decreasing.
\begin{proof}
Given $\hat \pi\in \hM(\hat \mu, \nu)$, write $\pi$ for the corresponding martingale transport $\pi\in \M( \mu, \nu)$ and $(\bar \pi_x)_x$ for its disintegration wrt $\mu$. We note that $\hat\pi$ can be $\mu$-a.s.\ uniquely represented in the form
\begin{align}
\hat \pi(A\times B\times C)= \int\dd\mu(x) \int \dd\alpha_x(u,x',z)\, 
\I_{A\times B\times C}(u,x,y),
\end{align}
where $(\alpha_x)$ is a (measurable) family with $\alpha_x\in \hM(\bar\mu_x\times \delta_x,\bar\pi_x).$
 We then find
\begin{align}
\inf_{\hat \pi\in \hM(\hat\mu, \nu)}\int (1-u)\sqrt{1+y^2}\, \dd\hat\pi&\, =
\inf_{\pi\in \M(\mu, \nu)}\inf_{\alpha\in \hM(\bar\mu_x\times \delta_x,\bar \pi_x)} \iint  (1-u)\sqrt{1+y^2}\, \dd\alpha(u, x', y) \, \dd\mu(x)\\ 
&\,=  \inf_{\pi\in \Pi(\mu, \nu)} \int C^{\hat \mu} (x,\bar\pi_x)\, \dd\mu(x).\qedhere
\end{align}
\end{proof}

\bibliographystyle{abbrv}
\bibliography{joint_biblio}

\end{document}